\documentclass{article}
\usepackage[english]{babel}
\usepackage{geometry,amsmath,amssymb,stmaryrd,latexsym,theorem}
\catcode`\<=\active \def<{
\fontencoding{T1}\selectfont\symbol{60}\fontencoding{\encodingdefault}}
\catcode`\>=\active \def>{
\fontencoding{T1}\selectfont\symbol{62}\fontencoding{\encodingdefault}}
\newcommand{\assign}{:=}
\newcommand{\backassign}{=:}
\newcommand{\cdummy}{\cdot}
\newcommand{\mathd}{\mathrm{d}}
\newcommand{\nin}{\not\in}
\newcommand{\nobracket}{}
\newcommand{\of}{:}
\newcommand{\textdots}{...}
\newcommand{\tmabbr}[1]{#1}
\newcommand{\tmaffiliation}[1]{\\ #1}
\newcommand{\tmem}[1]{{\em #1\/}}
\newcommand{\tmemail}[1]{\\ \textit{Email:} \texttt{#1}}
\newcommand{\tmmathbf}[1]{\ensuremath{\boldsymbol{#1}}}
\newcommand{\tmop}[1]{\ensuremath{\operatorname{#1}}}
\newcommand{\tmsep}{, }
\newcommand{\tmstrong}[1]{\textbf{#1}}
\newcommand{\tmtextit}[1]{{\itshape{#1}}}

\newcommand{\ack}{\flushleft\textbf{Acknowledgements.}\quad}

\newenvironment{proof}{\noindent\textbf{Proof\ }}{\hspace*{\fill}$\Box$\medskip}

\newtheorem{theorem}{Theorem}[section]

\newtheorem{axiom}[theorem]{Assumption}
\newtheorem{definition}[theorem]{Definition}
\newtheorem{lemma}[theorem]{Lemma}
{\theorembodyfont{\rmfamily}\newtheorem{remark}[theorem]{Remark}}

\newcommand{\tmkeywords}{\textbf{Keywords:} }
\newcommand{\tmmsc}{\textbf{A.M.S. subject classification:} }

\newcommand{\suppressed}[1]{\, #1 \,}
\newcommand{\explicitspace}{ }

\begin{document}

\title{Rough invariance principle for delayed regenerative processes}

\author{
  Tal Orenshtein
  \tmaffiliation{Technische Universit{\"a}t Berlin and Weierstrass Institute\\
  Strasse des 17. Juni 136, 10623 Berlin and Mohrenstrasse 39 10117 Berlin}
  \tmemail{orenshtein@wias-berlin.de}
}

\maketitle

\begin{abstract}
We derive an invariance principle for the lift to the rough path topology of stochastic processes with delayed regenerative increments under an optimal moment condition. An interesting feature of the result is the emergence of area anomaly, a correction term in the second level of the limiting rough path which is identified as the average stochastic area on a regeneration interval. A few applications include random walks in random environment and additive functionals of recurrent Markov chains. The result is formulated in the p-variation settings, where a rough Donsker Theorem is available under the second moment condition. The key renewal theorem is applied to obtain an optimal moment condition.
\end{abstract}

\tmmsc{60K37}{\tmsep}{ 60F17}{\tmsep}{ 82C41}{\tmsep}{ 82B43}

\tmkeywords{Invariance principle}{\tmsep}{rough paths}{\tmsep}{$p$-variation}{\tmsep}{area anomaly}{\tmsep}{regenerative process}{\tmsep}{key
renewal theorem}{\tmsep}{random walks in random environment.}

\section{Introduction}

Donsker's invariance principle states that a diffusively
rescaled centered random walk on $\mathbb{R}^d$ with jumps of finite variance
converges in distribution to a Brownian motion in the Skorohod topology. Regenerative
processes are a more general class: they are assumed to contain an infinite
subsequence of times on which the induced process is a random walk. The
natural strategy to prove an invariance principle here is to first prove it for that
subsequence, and then to show that the fluctuations of the
original process coincide with the ones of the approximating sequence in the limit. However,
when lifting regenerative processes to the rough path space a surprising
feature appears in the limiting rough path. The first level is naturally, the Brownian motion defined by the
covariance matrix achieved in the classical case, whereas the second level (see Section \ref{subseq:elements} for more
details) does not coincide with the iterated integral of the Brownian motion,
but should be corrected by a deterministic process. The latter is called
`area anomaly' and is linear in time and identified in terms of the
stochastic area on a regeneration time interval, see for example
{\eqref{eq:area:anomaly}} below. This provides non-trivial and new information on the
limiting path which is not captured in the classical invariance principle.
This information is crucial in order to describe the limit of stochastic differential equations (SDE) of the form
\[
 Y_t^{(n)} = Y_0 + \int_0^t b(Y_s^{(n)}) \mathd s +      \int_0^t \sigma(Y_s^{(n)}) \mathd X_s^{(n)},\quad t\in [0,T],
\]
where the driver $X^{(n)}$ is a linearly interpolated rescaled regenerative process, $b$ and $\sigma$ are smooth functions and the integral is in the sense of Riemann-Stieltjes. Even though $X^{(n)}$ converges weakly to a Brownian motion $B$, the limit of $Y^{(n)}$ does not satisfy the SDE with $b$ and $\sigma$ driven by $B$, but the is an additional drift which is explicit in terms of the area correction of the second level of the limiting rough path (denoted by $\Gamma$ in \eqref{eq:area:anomaly} below), see e.g \cite{DOP} line (1.1) and the discussion below it.

In this work we optimize the moment condition on
regeneration intervals that was assumed in {\cite{LO18}}. The main obstacle of {\cite{LO18}} is that it relies on the rough path extension of Donsker's Theorem
{\cite{breuillard2009random}} which is based on Kolmogorov's
tightness criterion on H{\"o}lder rough paths. As mentioned already in
{\cite{breuillard2009random}}, this was costly and therefore the assumed moment
condition was not optimal.
Instead of considering the somewhat heavy algebraic framework in H{\"o}lder rough path formalism,
we consider the parametrization-free $p$-variation settings which fits
better to jump processes and discrete-time processes.

Recently, a new machinery was
introduced to deal with regularity for jump processes in the rough path
topology. The main tool is the L{\'e}pingle Burkholder-Davis-Gundy (BDG) inequality lifted to the p-variation rough path settings.  The latter provides an equivalence
between the $(2 q)$-th moment of the $p$-variation norm of a local martingale
and the $q$-th moment of its quadratic variation, and allows to obtain the
rough version of Donsker's Theorem under the second moment condition, as in
the classical case. The proof is by now standard, however for completeness the details are included in Appendix \ref{app:donsker}.

In order to optimize the moment condition on a regeneration interval
for processes with regenerative increments so that the rough path version of Donsker's Theorem is used under not more than the second moment, we apply the Key Renewal Theorem. This theorem roughly says that the mass function of the process in a regeneration
interval around a fixed time (also called `age' in the renewal theory jargon)
is approaching a density which is proportional to the uniform measure of an
independent copy of the interval, namely, its size-biased version. This result
is sharp and in particular guarantees that in the limit as time goes to
infinity the $m$-th moment of a regeneration interval around a deterministic
time, and the $(m + 1)$-st moment of a fixed regeneration interval are equal
up to a constant. The result extends
the classical Donsker Theorem for these processes with no extra regularity
assumption.

There has been some progress related to random walks in the rough topology in
the past two decades. The closest works, generalized in this paper are
{\cite{LO18,lopusanschi2017levy,lopusanschi2017area}}. In the context of
semimartingales and rough paths with jumps
{\cite{chevyrev2018random,FZ18,CF19}}, CLT on nilpotent covering graphs and
crystal lattices
{\cite{ishiwata2018central,namba2017remark,ishiwata2018centralparttwo}},
additive functional of Markov process and random walks in random environment
{\cite{DOP}}. For homogenization in the continuous settings
{\cite{coutin2005semi,lejay2003importance,lejay2002convergence}}, and for additive functionals of fractional random fields
{\cite{gehringer2019homogenization,gehringer2020functional,gehringer2020rough}}.

In the remaining part of this section we shall present the model and the main
result, Theorem \ref{thm:main}, after introducing the necessary rough path
theory elements in Section \ref{subseq:elements}. In Chapter 2 we shall
mention some examples to which our main result applies, whereas its proof is
given in Chapter 3. We also included three short appendices which might be
useful in other context.

\subsection{Preliminaries on rough paths}\label{subseq:elements}
For two families $(a_i)_{i \in I}, (b_i)_{i \in I}$ of real numbers indexed
  by $I$, we write $a_i \lesssim b_i$ if there is a positive constant $c$ so
  that $a_i \leqslant cb_i$ for all $i \in I$. We write $a_n \approx b_n$
  whenever $a_n - b_n \rightarrow 0$ as $n \rightarrow \infty$.
  Set $\mathbb{N}=\{1,2,...\}$, $\mathbb{N}_0=\mathbb{N}\cup\{0\}$ and $\Delta_T
  \assign \{ (s, t) \of 0 \leqslant s < t \leqslant T \}$ for $T > 0$.
  For a function $X : [0, T] \rightarrow \mathbb{R}^d$ we set $X_{s, t}
  \assign X_t - X_s$. We interpret $X$ also as a function on $\Delta_T$ given
  by $(s, t) \mapsto X_{s, t}$. For a metric space $E$ we write $C ([0,
  T], E)$ resp. $D ([0, T], E)$) for the $E$-valued continuous resp.\
  c{\`a}dl{\`a}g functions on $[0, T]$. We write $C (\Delta_T, E)$ resp. $D
  (\Delta_T, E)$) for the space of $E$-valued functions $X : \Delta_T \rightarrow E$ so
  that $t \mapsto X_{s, t}$ is continuous resp. c{\`a}dl{\`a}g on $[s, T]$,
  for every $s \in [0, T]$. By convention, whenever $E =\mathbb{R}^d,$ we write
  $| \cdummy |$ for the $d$-dimensional Euclidean norm.
  Also, the expectation of a vector (or matrix) valued random variable is
  understood coordinate- (or entry-) wise. For $x, y \in \mathbb{R}^d$ we
  write $x \otimes y \in \mathbb{R}^{d \times d}$ for the tensor product $(x
  \otimes y)_{i, j} = x_i y_j, i, j = 1, \ldots, d$, and $x^{\otimes 2}$ for $x \otimes x$. Whenever $(x_n)_{n}$ is a sequence of elements in $\mathbb{R}^d$ we write $x_n^i,i=1,...,d$, for their components.

For brevity, we shall focus on the necessary objects needed for introducing our results. We 
follow closely Chapter 2 of {\cite{DOP}}. The reader
is referred to Section $5$ of {\cite{FZ18}} for details on It{\^o}
$p$-variation rough paths with jumps and to {\cite{friz2014course}} and
{\cite{friz2010multidimensional}} for an extensive account of the theory of rough paths.

For a normed space $(E, \|
\cdummy \|_E)$, and a function $X \in C ([0, T], E)$ or $X \in C (\Delta_T,
E)$ we write $\| X \|_{\infty, [0, T]} \assign \sup_{(s, t) \in \Delta_T} \|
X_{s, t} \|_E $ to denote the {\tmem{uniform norm}} of $X$, and for any $p\in(0, \infty)$, we write $\| X \|_{p, [0, T]} \assign \left( \sup \sum_{[s, t]
\in \mathcal{P}} \| X_{s, t} \|^p_E \right)^{1 / p}$, where the supremum is
over all finite partitions $\mathcal{P}$ of $[0, T]$, to denote its
{\tmem{$p$-variation norm}}. A continuous rough path is a pair of
functions $(X, \mathbb{X}) \in C ([0, T], \mathbb{R}^d) \times C (\Delta_T,
\mathbb{R}^{d \times d})$ satisfying Chen's relation, that is
\begin{equation}\label{eq:Chens}
  \mathbb{X}_{s, t} -\mathbb{X}_{s, r} -\mathbb{X}_{r, t} = X_{s, r} \otimes
  X_{r, t} \text{\quad for all } \suppressed{\explicitspace} 0 \leqslant s < r
  < t \leqslant T.
\end{equation}

\begin{definition}[$p$-variation rough path space]
  For $p \in [2, 3)$, $C_p ([0, T], \mathbb{R}^d \times \mathbb{R}^{d
  \times d})$ is the space of all continuous rough paths $(X, \mathbb{X})$   such that
  \begin{equation}\label{eq:p-var norm of pair}
    \interleave (X, \mathbb{X}) \interleave_{p, [0, T]} \assign | X_0 | + \| X
    \|_{p, [0, T]} + \| \mathbb{X} \|_{p / 2, [0, T]} < \infty .
  \end{equation}
  The (uniform) {\tmem{$p$-variation distance}} $\sigma_{p, [0, T]}$
  is defined by taking the norm of the path defined by the diferences
  increment-wise:
  \[ \sigma_{p, [0, T]} ((X, \mathbb{X}), (Y, \mathbb{Y})) \assign
     \interleave (X - Y, \mathbb{X}-\mathbb{Y}) \interleave_{p, [0, T]} . \]
\end{definition}

We refer to $X$ as the first level of the rough path $(X, \mathbb{X})$ and to
$\mathbb{X}$ as its second level. We shall now state a simple and useful
sufficient condition for convergence in $p$-variation based on the convergence
in the uniform topology together with tightness of the $p$-variation norms,
see Lemma 2.3 in {\cite{DOP}} and which is based on Theorem 6.1 of
{\cite{FZ18}}.

\begin{lemma}[Sufficient condition for convergence in $p$-variation]
  \label{lem:tightness of norms is enough}Assume that \linebreak $(Z_n, \mathbb{Z}_{n})_{n\in\mathbb{N}}$ is
  a sequence of continuous rough paths and let $ p_0\in (2,3)$. Assume also that
  there exists a continuous rough path $(Z, \mathbb{Z})$ such that $(Z_n,
  \mathbb{Z}_n) \rightarrow (Z, \mathbb{Z})$ in distribution in the uniform
  topology and that the family of real valued random variables $(\| (Z_n,
  \mathbb{Z}_n) \|_{p_0, [0, T]})_{n \in \mathbb{N}}$ is tight. Then $(Z_n,
  \mathbb{Z}_n) \rightarrow (Z, \mathbb{Z})$ in distribution in the
  $p_0$-variation uniform topology $C_p ([0, T], \mathbb{R}^d \times
  \mathbb{R}^{d \times d})$ for all $p \in (p_0, 3)$.
\end{lemma}

\subsection{Main result}

Let $X = (X_n)_{n \in \mathbb{N}_0}$ be a discrete time stochastic process on
$\mathbb{R}^d$ defined on a probability space $(S, \mathcal{F}, \mathbb{P})$
and let $\mathbb{E}$ be the corresponding expectation. Assume that $X$ has a
delayed regenerative increments, that is there exists a sequence $0 = : \tau_0
< \tau_1 < \tau_2 <${\textdots} of  \linebreak $\mathcal{F}$-measurable
$\mathbb{N}_0$-valued random variables so that $(T_k, \{X_{\tau_k, \tau_k
+ m}, 0 \leqslant m \leqslant T_k \})_{k \in\mathbb{N}}$ is an i.i.d.\ family independent of
$(T_0, \{X_{0, m}, 0 \leqslant m \leqslant \tau_1 \})$ under $\mathbb{P}$, where $T_k =
\tau_{k + 1} - \tau_k$ are the {\tmem{regeneration intervals}} and $X_{\ell,
k} \assign X_k - X_{\ell}$ are the increments. Assume that $\mathbb{E}
[X_{\tau_1, \tau_2}] = 0$ and that $\mathrm{gcd} \{ j : p_j > 0 \} = d$ for
some $d \in \mathbb{N}$, where $p_j =\mathbb{P} (T_1 = j)$, that is, $T_1$ is
$d$-arithmetic. For any sequence $(X_k)_{k \in \mathbb{N}_0}$ of elements in
$\mathbb{R}^d$ we define
\begin{align}
  X^{(n)}_t & \assign  \frac{1}{\sqrt{n}} X_{\lfloor nt \rfloor} + \frac{nt -
  \lfloor nt \rfloor}{\sqrt{n}}  (X_{\lfloor nt \rfloor + 1} - X_{\lfloor nt
  \rfloor}) \quad\text{ and } \nonumber\\
  \mathbb{X}^{(n)}_{s, t} & \assign  \frac{1}{n} I^{\mathrm{Str}}_{\lfloor ns
  \rfloor, \lfloor nt \rfloor} (X) + \frac{n (t - s) - \lfloor nt \rfloor +
  \lfloor ns \rfloor}{n}  \big( I^{\mathrm{Str}}_{\lfloor ns \rfloor, \lfloor
  nt \rfloor + 1} (X) - I^{\mathrm{Str}}_{\lfloor ns \rfloor, \lfloor nt
  \rfloor} (X) \big), \label{eq:interpolation}
\end{align}
where for positive integers $M < N$

\begin{align*}
  I^{\mathrm{Str}}_{M, N} (X) & \assign \sum_{M + 1 \leqslant k \leqslant N} \left( X_{M, k - 1}
  \otimes X_{k - 1, k} - \frac{1}{2} X_{k - 1, k} \otimes X_{k - 1, k} \right)
  .
\end{align*}

\begin{remark}
  Note that
  \[ \mathbb{X}_{s, t}^{(n)} \assign \int_{(s, t]} \int_{(s, u)} \mathd
     X_v^{(n)} \otimes \mathd X_u^{(n)}, \]
  where the integration with respect to $\mathrm{d} X^{(n)}_v$ is in the sense
  of Riemann-Stieltjes.
\end{remark}

We shall now formulate the main regularity assumption. We remind the reader that for $Y\in\mathbb{R}^d$ we write $Y_k^i$ for the $i$-th component of $Y$, $i\in\{1,...,d\}$.
\begin{axiom}
  \label{ass:jumps}For all $ i\in\{1,...,d\}$, $m \in\{ 0, 1\}$ and $p \in
  \{ 0, 2 \}$
  \[ 0 < \mathit{{\tmstrong{{\tmem{}}}}} \mathbb{E} \bigg[\big( \Xi_m^i \big)^p T_k\bigg] <
     \infty, \]
  where $\Xi_m^i \assign \sup \{ | X^i_{\tau_m, \tau_m + k} | : 0 \leqslant k \leqslant   T_m \}$.
\end{axiom}

\begin{theorem}
  \label{thm:main}Let $X$ be a discrete time stochastic process satisfying
  Assumption \ref{ass:jumps}. Assume that $\mathit{{\tmstrong{{\tmem{}}}}}
  \mathbb{E} [X_{\tau_k, \tau_{k + 1}}] = 0$ for every $k \in \mathbb{N}_0.$ Then,
$(X^{(n)}, \mathbb{X}^{(n)})_{n\in\mathbb{N}}$ converges in distribution to $(B, \mathbb{B}+ \cdummy
  \Gamma)$ in $C_p ([0, T], \mathbb{R}^d \times \mathbb{R}^{d \times
  d})$ for every $p > 2$, where $B$ is a centered Brownian motion with
  covariance
  \begin{equation}\label{eq:cov_B}
    [B, B]_t = \frac{\mathbb{E} [X_{\tau_1, \tau_2}^{\otimes 2}]}{\mathbb{E}
    [T_1]} t,
  \end{equation}
  $\mathbb{B}$ is the Stratonovich iterated integral of $B$,
  that is $\mathbb{B}_{s, t} = \int_{(s, t]} B_{s, u} \otimes \circ \mathrm{d} B_u$
  and $\Gamma \in \mathbb{R}^{d \times d}$ is given by
  \begin{equation}\label{eq:area:anomaly}
    \Gamma = \frac{\mathbb{E} [A_{\tau_1, \tau_2} (X)]}{\mathbb{E} [T_1]},
  \end{equation}
  where $A_{M, N} (X) = \mathrm{Antisym} (\mathbb{X}^{(1)}_{M, N})$ is the
  antisymmetric part of the matrix $\mathbb{X}^{(1)}_{M, N}$ . The notation
  $\mathbb{B}+ \cdummy\Gamma$ above is for $\big(\mathbb{B}_{s,t}+ (t-s)\Gamma\big)_{(s,t)\in\Delta_T}$.
\end{theorem}

\begin{remark}
  Note that the positivity condition for the moment in Assumption
  \ref{ass:jumps} is assumed in order to avoid degeneracies. Indeed, it can be
  omitted if we accept a degenerate formulation of the invariance principle.
  More accurately, if this is violated, then $X_n^i = 0$ for all times for
  some coordinate $i$, for which one can say that the invariance principle
  holds with a singular covariance matrix. Note also that Assumption
  \ref{ass:jumps} holds whenever $\tau_1$ has a third moment and $\frac{|
  X_{\ell, k} |}{k - \ell}$ are uniformly bounded from above by a constant
  a.s., for example, whenever the process $X$ has values in $\mathbb{Z}^d$,
  with nearest-neighbor jumps.
\end{remark}

\begin{remark}[Optimality of the result]
  Let $X_n = \sum_{k = 1}^n \xi_k$ be a centered random walk on
  $\mathbb{R}^d$, that is $\mathbb{E} [\xi_k^i] = 0, i = 1, \ldots, d$, where
  $\xi_k = (\xi_k^1, \ldots, \xi_k^d)$ are $\mathbb{R}^d$-valued i.i.d random
  variables. Then $X = (X_n)_{n \in \mathbb{N}_0}$ has, trivially, delayed
  regenerative increments: here $T_k = 1$ and $\Xi_k^i = | \xi_k^i |$, $i = 1,
  \ldots, d,$ $k = 0, 1, \ldots$ Therefore Assumption \ref{ass:jumps} is
  equivalent in this case to $\mathbb{E} \left[ \sum_{i = 1}^d | \xi_1^i |^2
  \right] < \infty$, which is a necessary and sufficient condition for the
  classical central limit theorem, cf. Theorem 4 in Chapter 7 of   {\cite{gnedenko1968limit}}.
  \end{remark}

\section{Examples}

\paragraph{Positive recurrent countable Markov chains.}Assume that $(Y_k)_{k\in\mathbb{N}}$ is a
positive recurrent irreducible Markov chain taking values in some measurable space
$\tmmathbf{S}$, that is
\[ \mathbb{P} (T^+_x < \infty |X_0 = x) = 1 \tmop{for} \tmop{some} x \in
   \tmmathbf{S} \]
where $T^+_x = \inf \{ k \in \mathbb{N} \of Y_k = x \}$. Assume moreover that
$\mathbb{E} [(T^+_x)^3 |X_0 = x] < \infty$. Then the conditions of Theorem
\ref{thm:main} hold under $\mathbb{P} (\cdot |X_0 = x)$ for the sequence
\[ X_n \assign \sum_{k = 0}^n f (Y_k) - n \frac{\mathbb{E} [D|Y_0 =
   x]}{\mathbb{E} [T^+_x |Y_0 = x]}, n \in \mathbb{N}_0, \]
where $D = \sum_{k = 0}^{T^+_x - 1} f (Y_k)$, and $f : \tmmathbf{S}
\rightarrow \mathbb{R}^d$ is any bounded measurable function.
\

All the examples considered in Section 5 of {\cite{LO18}} are applicable
here. In particular, the result applies to {\tmstrong{Random walks in periodic
environment}} ({\cite{LO18}}, Section 5.2), where the periodicity assumption
can be easily relaxed. For example, it can include i.i.d.\ impurities, as long
as the regenerative structure is kept. It applies also to {\tmstrong{Random
walks on covering graphs and hidden Markov chains}}, (Chapter 5.2 of
{\cite{LO18}}). The examples there immediately satisfy Assumption
\ref{ass:jumps}. Moreover, these can be extended, allowing infinite modulating
systems and covering graphs with infinite structure. Another application is to
the so-called {\tmstrong{Ballistic random walks in random environment}}
(Section 5.1 of {\cite{LO18}}). Lastly, for the example of {\tmstrong{Random
walks in Dirichlet environments}}, the rough invariance principle, Theorem 5.5
of {\cite{LO18}} is shown in the ballistic regime (more accurately, whenever a condition which is denoted by
$(T)_{\gamma}$ holds for some $\gamma \in (0, 1)$, see Chapter 6.1 of
{\cite{sabot17}} for teh definition and more details) in the sense of H{\"o}lder rough paths only
whenever the trap parameter $\kappa$ satisfies $\kappa > 8$ in the
$\alpha$-H{\"o}lder rough path topology for all $\alpha < \frac{1}{2} -
\frac{1}{(\kappa / 2)^{\ast}}$, where $(\kappa / 2)^{\ast} = \min \{ \lfloor
\kappa / 2 \rfloor, 2 \lfloor \kappa / 4 \rfloor \}$. The improvement in the
present work is that the rough invariance principle - Theorem \ref{thm:main}
below - applies already whenever the trap parameter satisfies $\kappa > 3$,
and moreover it holds for all $p > 2$, which corresponds to all $\alpha < 1 / 2$
in the $\alpha$-H{\"o}lder settings.

\section{Proof of Theorem \ref{thm:main}}

  Set $Z_k \assign X_{\tau_k} = \sum_{\ell = 0}^{k - 1} Y_\ell$ for $k \ge 0$, where $Y_\ell
  \assign X_{\tau_{\ell }, \tau_{\ell+ 1}}$. Then $Z = (Z_k)_{k \in \mathbb{N}_0} $ is a random walk with square integrable jumps. In particular,
  \begin{equation}\label{eq:IPforZ}
    (Z^{(n)}, \mathbb{Z}^{(n)})  \text{ converges in distribution to } (B^Z, \mathbb{B}^Z) \text{ in } C_p ([0, T], \mathbb{R}^d)
  \end{equation}
  for all $p > 2$, where $B^Z$ is a centered Brownian motion with covariance
  \[ [B^Z, B^Z]_t =\mathbb{E} [Y_2 \otimes Y_2] t \]
  and $\mathbb{B}^Z$ is the Stratonovich iterated integral of $B^Z$.

  By Lemma \ref{lem:tightness of norms is enough} it
  is enough to prove first convergence in distribution in the uniform topology and then to
  show tightness for the sequence of $p$-variation norms.

  We shall work on each level separately. Let us first identify the limit by
  proving the convergence of the finite-dimensional distributions. For ease of
  notation we shall only show the one-dimensional distributions, proving the
  convergence for higher dimensions is done similarly. For any $u \geqslant 0$ let $\kappa (u)$ be the
  unique random integer $k$ so that $\tau_k \leqslant u < \tau_{k + 1}$. Note that
  $\kappa (u)$ is measurable with respect to $\sigma (T_k \of k < u)$, since
  $\tau_k = \sum_{\ell = 0}^{k - 1} T_{\ell}$ . Observe that
  \[ | X^{(n), i}_t - Z^{(n), i}_{\kappa (nt) / n} | = \left|
     \frac{1}{\sqrt{n}} X_{\lfloor nt \rfloor}^i + \frac{nt - \lfloor nt
     \rfloor}{\sqrt{n}}  (X_{\lfloor nt \rfloor + 1}^i - X_{\lfloor nt
     \rfloor}^i) - \frac{1}{\sqrt{n}} X_{\tau_{\kappa (nt)}}^i \right|
     \leqslant \frac{1}{\sqrt{n}} | \Xi^i_{\kappa (nt)} | . \]
  Next, note that as $\tau_{\kappa (nt)} \leqslant nt < \tau_{\kappa (nt) + 1}$
  \[ \frac{\kappa (nt)}{\kappa (nt) + 1} \frac{\kappa (nt) + 1}{\tau_{\kappa
     (nt) + 1}} \leqslant \frac{\kappa (nt)}{nt} \leqslant \frac{\kappa
     (nt)}{\tau_{\kappa (nt)}}. \]
  The weak law of large numbers for $(\tau_k)_{k \in \mathbb{N}_0}$
  implies that
  \begin{equation}\label{eq:renewal-thm}
    \frac{\kappa (nt)}{nt} \xrightarrow[n \rightarrow \infty]{} \mathbb{E}
    [T_1]^{- 1} \backassign \beta
  \end{equation}
  in probability with respect to $\mathbb{P}$, where we used the fact that $1
  \leqslant \mathbb{E} [T_1] < \infty$ by assumption. In particular,
  $\mathbb{P} (\kappa (n) > 2 \beta n) \xrightarrow[n \rightarrow \infty]{} 0$
  and therefore for every $\varepsilon>0$
  \begin{eqnarray*}
    \mathbb{P} (\| X^{(n), i} - Z^{(n), i}_{\kappa (n \cdot) / n}
    \|_{\infty, [0, T]} > \varepsilon) & \leqslant & \mathbb{P} \left( \|
    \Xi^i_{\kappa (n \cdot)} \|_{\infty, [0, T]} > \varepsilon \sqrt{n}
    \right)\\
    & \leqslant & \mathbb{P} \left( \max_{0 \leqslant m \leqslant nT 2 \beta}
    \Xi^i_m > \varepsilon \sqrt{n} \right) + o (1) .
  \end{eqnarray*}
  But since the maximum of order $n$ i.i.d random variables with a finite
  second moment is sub-diffusive in probability the first term also vanishes
  in probability. Indeed,
  \begin{eqnarray*}
    \mathbb{P} \left( \max_{0 \leqslant m \leqslant cn} \Xi^i_m > \varepsilon
    \sqrt{n} \right) & = & 1 - \left( 1 -\mathbb{P} \left( \Xi^i_0 >
    \varepsilon \sqrt{n} \right) \right) \left( 1 -\mathbb{P} \left( \Xi^i_1 >
    \varepsilon \sqrt{n} \right) \right)^{\lfloor cn \rfloor}\\
    & \approx & 1 - \left( 1 -\mathbb{P} \left( \Xi^i_1 > \varepsilon
    \sqrt{n} \right) \right)^{\lfloor cn \rfloor}\\
    & \approx & 1 - \exp \left( - cn\mathbb{P} \left( \Xi^i_1 > \varepsilon
    \sqrt{n} \right) \right),
  \end{eqnarray*}
  which tends to 0 as $n\to\infty$. This holds since $\mathbb{P} (| \Xi^i_1 |^2 > \varepsilon^2 n)
  \leqslant \mathbb{P} (| \Xi^i_1 |^2 > \varepsilon^2 j)$ for every $j
  \leqslant n$ and so
  \[ (n - k) \mathbb{P} (| \Xi^i_1 |^2 > \varepsilon^2 n) \leqslant
     \sum^n_{j = k + 1} \mathbb{P} (| \Xi^i_1 |^2 > \varepsilon^2 j), \quad
     k \leqslant n. \]
  Therefore,
  \[ \limsup_{n \rightarrow \infty} n\mathbb{P} \left( \Xi^i_1 > \varepsilon
     \sqrt{n} \right) = \limsup_{n \rightarrow \infty}  (n - k) \mathbb{P}
     (| \Xi^i_1 |^2 > \varepsilon^2 n) \leqslant \sum^{\infty}_{j = k + 1}
     \mathbb{P} (\varepsilon^{- 2} | \Xi^i_1 |^2 > j) \xrightarrow[k
     \rightarrow \infty]{} 0, \]
  since the right hand side is summable as $\mathbb{E} [| \Xi^i_1 |^2] <
  \infty$.

  Next, since the maximum of linear interpolations of any finite sequence on
  any bounded interval is obttained on the end points, we have
  \begin{eqnarray*}
    \| Z_{\kappa (n \cdot) / n}^{(n), i} - Z_{\cdot \beta}^{(n), i}
    \|_{\infty, [0, T]} & \leqslant & \frac{1}{\sqrt{n}} \max_{0 \leqslant m
    \leqslant T \beta n}  | Z_{\kappa (m / \beta)}^i - Z_m^i |\\
    & \leqslant & \frac{1}{\sqrt{n}} \max_{0 \leqslant m \leqslant Tn \beta}
    | \Xi^i_m | \max_{0 \leqslant m \leqslant Tn \beta} | \kappa (m / \beta) -
    m | .
  \end{eqnarray*}
    Thus for $R>0$
  \begin{align*}
    & \mathbb{P} (\| Z_{\kappa (n \cdot) / n}^{(n), i} - Z_{\cdot \beta}^{(n),
    i} \|_{\infty, [0, T]} > \varepsilon)    \\
    &  \quad \leqslant \mathbb{P} \left( \max_{0 \leqslant m \leqslant T \beta n}
    | Z_{\kappa (m / \beta)}^i - Z_m^i | > \varepsilon \sqrt{n}, \max_{0
    \leqslant m \leqslant Tn \beta} | \kappa (m / \beta) - m | \leqslant R
    \sqrt{n} \right)  \\
    &  \qquad +\mathbb{P} \left( \max_{0 \leqslant m \leqslant T \beta n}  |
    Z_{\kappa (m / \beta)}^i - Z_m^i | > \varepsilon \sqrt{n}, \max_{0
    \leqslant m \leqslant Tn \beta} | \kappa (m / \beta) - m | > R \sqrt{n}
    \right) \\
    & \quad \leqslant \mathbb{P} \left( \max_{0 \leqslant k, m \leqslant T \beta
    n, | k - m | \leqslant R \sqrt{n}}  | Z_k^i - Z^i_m | > \varepsilon
    \sqrt{n} \right) +\mathbb{P} \left( \max_{0 \leqslant m \leqslant Tn
    \beta} | \kappa (m / \beta) - m | > R \sqrt{n} \right)  \\
    & \quad \leqslant 3 T \beta \sqrt{n} \mathbb{P} \left( \max_{k \leqslant
    R \sqrt{n}}  | Z_k^i | > \varepsilon \sqrt{n} \right) +\mathbb{P}
    \left( \max_{0 \leqslant m \leqslant Tn \beta} | \kappa (m / \beta) - m |
    > R \sqrt{n} \right) .
  \end{align*}
  To deal with the first term we use a standard two pairs estimate (see
  Example 10.1 of Billingsley {\cite{Bil99}}), to find some $K > 0$ and $\Psi
  (\lambda) \xrightarrow[\lambda \rightarrow \infty]{} 0$ so that for any fixed $R>0$
  \begin{eqnarray}
    \sqrt{n} \mathbb{P} \left( \max_{k \leqslant R \sqrt{n}}  | Z_k^i | >
    \varepsilon \sqrt{n} \right) & \leqslant & \sqrt{n} \frac{K}{\left(
    \varepsilon n^{1 / 4} / \sqrt{R} \right)^4} + \sqrt{n}  \frac{\Psi
    \left( \frac{\varepsilon \sqrt{n}}{\sqrt{R}} \right)}{\left( \varepsilon
    n^{1 / 4} / \left( 2 \sqrt{R} \right) \right)^2}
    \label{eq:two:pairs:estimate} \\
    & = & \frac{KR^2}{\varepsilon^4 \sqrt{n}} + \frac{4 R}{\varepsilon^2}
    \Psi \left( \frac{\varepsilon \sqrt{n}}{\sqrt{R}} \right) \xrightarrow[n
    \rightarrow \infty]{} 0. \nonumber
  \end{eqnarray}
  By the central limit theorem for renewal processes with
  finite variance
  \[ \mathbb{P} \left( \max_{0 \leqslant m \leqslant Tn \beta} | \kappa (m /
     \beta) - m | > R \sqrt{n} \right) \xrightarrow[n \rightarrow \infty]{}
     \mathbb{P} (\max_{0 \leqslant t \leqslant T} c | W_t | > R),
     \label{eq:clt-renewal} \]
  where $W$ is a Brownian motion and $c > 0$ is some constant. As $R$ can be
  chosen arbitrarily large, $\limsup_{n \rightarrow \infty} \mathbb{P} (\|
  Z_{\kappa (n \cdot) / n}^{(n), i} - Z_{\cdot \beta}^{(n), i} \|_{\infty, [0,
  T]} > \varepsilon) = 0$. Therefore,
  \[ \| X^{(n), i} - Z^{(n), i}_{\cdot \beta} \|_{\infty, [0, T]}
     \xrightarrow[n \rightarrow \infty]{} 0 \qquad \tmop{in}
     \tmop{probability} . \]
  \label{eq:ZisClose}Applying Slutsky's Theorem in the Skorohod topology,
  Theorem \ref{thm:slutsky-skorohod}, to
  \begin{equation}
    X^{(n)} = X^{(n)} - Z_{\beta \cdot}^{(n)} + Z_{\beta \cdot}^{(n)}
  \end{equation}
  we deduce that $X^{(n)} \xrightarrow[n \rightarrow \infty]{} B_{\beta
  \cdot}^Z \backassign B$ in distribution with respect to $\mathbb{P}$, so
  that $B$ is a $d$-dimensional Brownian motion with a covariance matrix
  given in \eqref{eq:cov_B}, as desired.

  Next, we shall show tightness of the $p$-variation norms for $p > 2$. Note
  that $\kappa (u) \leqslant u$ for any $u \geqslant 0$. Indeed, $T_i\in \mathbb{N}$ for all $i \in \mathbb{N}$ by assumption and therefore $\kappa
  (u) \leqslant \tau_{\kappa (u)} \leqslant u$. Now, since
  \[ | X^{(n), i}_{s, t} - Z^{(n), i}_{\kappa (ns) / n, \kappa (nt) / n} |
     \leqslant \frac{1}{\sqrt{n}} (| \Xi^i_{\kappa (ns)} | + | \Xi^i_{\kappa
     (nt) + 1} |), \]
 for any partition $0 = t_0 < t_1 < \cdots < t_m = T$, we have
  \[ \sum_{r = 1}^m | X_{nt_{r - 1}, nt_r}^i - Z^i_{\kappa (nt_{r - 1}),
     \kappa (nt_r)} |^2 \lesssim \sum_{0 \leqslant k < nT} \# \{ 1 \leqslant r
     \leqslant m \of \kappa (nt_r) = k \}  | \Xi^i_k |^2 \leqslant \sum_{0
     \leqslant k < nT} T_k  | \Xi^i_k |^2 . \]
  Hence,
  \[ \mathbb{E} [\| X^{(n)} - Z^{(n)}_{\kappa (n \cdummy) / n} \|_{2, [0,
     T]}^2] \lesssim \frac{1}{n}  ((n - 1) \mathbb{E} [T_1  | \Xi_1 |^2]
     +\mathbb{E} [T_0  | \Xi_0 |^2]) \lesssim 1. \]
  Next, since $\kappa (nT) \leqslant nT$ we have that $\left(
  \frac{1}{\sqrt{n}} Z^i_{\kappa (\ell)} \right)_{\ell \leqslant nT}$ is a
  subsequence of $\left( \frac{1}{\sqrt{n}} Z^i_k \right)_{k \leqslant nT}$
  and therefore $\mathbb{E} [\| Z^{(n), i}_{\kappa (n \cdummy) / n} \|^2_{p,
  [0, T]}] \leqslant \mathbb{E} [\| Z^{(n), i} \|_{p, [0, T]}^2] \lesssim 1$,
  see \eqref{eq:tightness p-var first level}.

  Using the triangle inequality together with the fact that the $p$-variation
  norms are monotonically decreasing in $p$,
  \[ \mathbb{E} [\| X^{(n)} \|^2_{p, [0, T]}] \lesssim \mathbb{E} [\| X^{(n)}
     - Z^{(n), i}_{\kappa (n \cdummy) / n} \|_{2, [0, T]}^2] +\mathbb{E} [\|
     Z^{(n), i}_{\kappa (n \cdummy) / n} \|^2_{p, [0, T]}] \lesssim 1, \]
  as desired.

  We shall now treat the second level. Let us show first convergence in the
  Skorohod topology. Note that for $\ell < k$ we have the following
  decomposition which is a consequence of Chen's relation {\eqref{eq:Chens}}
  \begin{equation}\label{eq:chens for I}
  I^{\mathrm{Str}}_{\tau_{\ell}, \tau_k} (X) =
    I^{\mathrm{Str}}_{\ell, k}  (Z) + \sum_{\ell + 1 \leqslant u \leqslant k} A_{\tau_{u -
    1}, \tau_u}  (X) .
  \end{equation}
  Indeed, $I^{\mathrm{Str}}_{\tau_{u - 1}, \tau_u} (X) = \mathrm{sym}
  (I^{\mathrm{Str}}_{\tau_{u - 1}, \tau_u} (X)) + A_{\tau_{u - 1}, \tau_u}
  (X)$, and by a direct computation $\mathrm{sym} (I^{\mathrm{Str}}_{\tau_{u -
  1}, \tau_u} (X)) = \frac{1}{2} Z_{u - 1, u} \otimes Z_{u - 1, u}$.
  Therefore,

  \begin{align*}
    I^{\mathrm{Str}}_{\tau_{\ell}, \tau_k} (X) & = I^{\mathrm{Str}}_{\ell, k}
    (Z) + \sum_{\ell + 1 \leqslant u \leqslant k} I^{\mathrm{Str}}_{\tau_{u - 1}, \tau_u}
    (X) - \frac{1}{2}  \sum_{\ell + 1 \leqslant u \leqslant k} Z_{u - 1, u} \otimes Z_{u -
    1, u}\\
    & = I^{\mathrm{Str}}_{\ell, k} (Z) + \sum_{\ell + 1 \leqslant u \leqslant k}
    A_{\tau_{u - 1}, \tau_u} (X) .
  \end{align*}

  By Remark 2.2 of {\cite{DOP}}, it is enough to prove the convergence of
  $\mathbb{X}_t^{(n)} \assign \mathbb{X}_{0, t}^{(n)}$. By {\eqref{eq:chens for I}},
  we have for any $t \geqslant 0$
  \begin{eqnarray*}
    \frac{1}{n} \left| \mathbb{X}_{nt}^{(1)} -\mathbb{Z}_{\kappa (nt)} -
    \sum_{1 \leqslant u \leqslant \kappa (nt)} A_{\tau_{u - 1}, \tau_u} (X) \right| & = &
    \frac{1}{n}  | \mathbb{X}_{\tau_{\kappa (nt)}, nt}^{(1)} + X_{0,
    \tau_{\kappa (nt)}} \otimes X_{\tau_{\kappa (nt)}, nt} |\\
    & \leqslant & \frac{1}{n}  | \Xi_{\kappa (nt)}^{\otimes 2} | +
    \frac{1}{n}  | Z_{\kappa (nt)} | \otimes | \Xi_{\kappa (nt)} | .
  \end{eqnarray*}
  Therefore
  \begin{eqnarray*}
    \sup_{0 \leqslant s, t \leqslant T} \frac{1}{n} \left| \mathbb{X}_{ns,
    nt}^{(1)} -\mathbb{Z}_{\kappa (ns), \kappa (nt)} - \sum_{\kappa (ns) < u
    \leqslant \kappa (nt)} A_{\tau_{u - 1}, \tau_u} (X) \right| &  & \\
    \leqslant \quad 2 \sup_{0 \leqslant t \leqslant T} \frac{1}{n}  (|
    \Xi_{\kappa (nt)}^{\otimes 2} | + | Z_{\kappa (nt)} | \otimes |
    \Xi_{\kappa (nt)} |) . &  &
  \end{eqnarray*}
  But $\sup_{0 \leqslant t \leqslant T} \frac{1}{n}  (| \Xi_{\kappa
  (nt)}^{\otimes 2} | + | Z_{\kappa (nt)} | \otimes | \Xi_{\kappa (nt)} |)
  \xrightarrow[n \rightarrow 0]{} 0$ in probability. Indeed, we have already
  seen that $\mathbb{P} \left( \| \Xi^i_{\kappa (n \cdot)} \|_{\infty, [0, T]}
  > \varepsilon \sqrt{n} \right) \rightarrow 0$ for any $\varepsilon > 0$ and
  for the second term
  \begin{eqnarray*}
    \mathbb{P} (\sup_{0 \leqslant t \leqslant T} | Z_{\kappa (nt)} \nobracket
    \otimes \nobracket \Xi_{\kappa (nt)} | > \varepsilon n) & \leqslant &
    \mathbb{P} \left( \| Z_{n \cdot} \|_{\infty, [0, T]} > R \sqrt{n} \right)
    +\mathbb{P} \left( \max_{0 \leqslant k \leqslant nT} | \Xi_k | >
    \varepsilon \sqrt{n} / R \right) .
  \end{eqnarray*}
  But for any $R > 0$ the right term vanishes as $n \rightarrow \infty$, while
  the left term converges to $\mathbb{P} (\max_{0 \leqslant t \leqslant T} |
  cW_t | > R)$, where $W$ is a Brownian motion and $c > 0$ is a constant,
  which vanishes as $R \rightarrow \infty$.
  By Slutsky's Theorem the convergence of $\mathbb{X}^{(n)}$ in
  distribution holds if
  \[ \frac{1}{n} \mathbb{Z}_{\kappa (n \cdot)} + \frac{1}{n}  \sum_{1 \leqslant u
     \leqslant \kappa (n \cdot)} A_{\tau_{u - 1}, \tau_u}  (X) \xrightarrow[n
     \rightarrow \infty]{} \mathbb{B}_{0, \cdot} + \Gamma \cdot \]
  in distribution in the uniform topology. To achieve the last convergence,
  first note that $(A_{\tau_{u - 1}, \tau_u}  (X))_{u \in \mathbb{N}}$ are
  independent random matrices with the same law for $u \geqslant 2$. Note also
  that $| A_{\tau_{u - 1}, \tau_u}  (X) | \leqslant 4 | \Xi_{u - 1} |^{\otimes
  2} T_{u - 1}$ which implies that $\mathbb{E} [| A_{\tau_{u - 1}, \tau_u}
  (X) |] < C$ for $u = 1, 2$ by Assumption \ref{ass:jumps}. Hence the weak law
  of large numbers for the sum yields
  \[ \frac{1}{n}  \sum_{1 \leqslant k \leqslant n} A_{\tau_{k - 1}, \tau_k}  (X)
     \xrightarrow[n \rightarrow \infty]{} \mathbb{E} [A_{\tau_1, \tau_2}  (X)]
  \]
  in probability. Together with the convergence in probability $\frac{\kappa
  (nt)}{nt} \xrightarrow[n \rightarrow \infty]{} \mathbb{E} [T_1]^{- 1}$ which
  is the consequence of Assumption \ref{ass:jumps}. with $\alpha = 0$, we
  deduce that
  \[ \frac{1}{n}  \sum_{1 \leqslant k \leqslant \kappa (nt)} A_{\tau_{k - 1}, \tau_k}
     (X) \xrightarrow[n \rightarrow \infty]{} \frac{\mathbb{E} [A_{\tau_1,
     \tau_2}  (X)]}{\mathbb{E} [T_1]} \backassign \Gamma \quad \tmop{almost}
     \tmop{surely} . \]
  Moreover, since $\frac{\mathbb{E} [| A_{\tau_1, \tau_2}  (X) |]}{\mathbb{E}
  [T_1]} < \infty$ we have moreover
  \[ \left\|| \frac{1}{n}  \sum_{1 \leqslant k \leqslant \kappa (n \cdot)} A_{\tau_{k -
     1}, \tau_k}  (X) - \Gamma \cdot \right\||_{\infty, [0, T]} \xrightarrow[n
     \rightarrow 0]{} 0 \quad \tmop{in} \tmop{probability} . \]
  Using Slutsky's Theorem again together with {\eqref{eq:IPforZ}} it is left
  to show that
  \[ \| \mathbb{Z}_{\kappa (n \cdot) / n}^{(n)} -\mathbb{Z}_{\cdot
     \beta}^{(n)} \|_{\infty, [0, T]} \xrightarrow[n \rightarrow \infty]{}
     0 \text{ in probability} . \]
  As for the case of the first level we use {\eqref{eq:clt-renewal}} to reduce
  the last convergence to showing that $\| \mathbb{Z}_{\kappa (n \cdot) /
  n}^{(n)} -\mathbb{Z}_{\cdot \beta}^{(n)} \|_{\infty, [0, T]} 1_{\left\{
  \max_{0 \leqslant m \leqslant Tn} | \kappa (m) - m \beta | \leqslant R
  \sqrt{n} \right\}} \xrightarrow[n \rightarrow \infty]{} 0 \text{ in
  probability}$. Fix $\varepsilon > 0$.
  \begin{eqnarray*}
    \mathbb{P} \left( \max_{0 \leqslant k, m \leqslant T \beta n, | k - m |
    \leqslant R \sqrt{n}}  | \mathbb{Z}_{0, m}^{i, j} -\mathbb{Z}_{0,
    k}^{i, j} | > \varepsilon n \right) & \\
    \lesssim \quad \quad \mathbb{P} \left( \max_{0 \leqslant k, m \leqslant T
    \beta n, | k - m | \leqslant R \sqrt{n}}  | Z^i_{k, m} | \right. & |
    Z^j_{0, k} | > \varepsilon n \nobracket)\\
    \leqslant \quad \quad \mathbb{P} \left( \max_{0 \leqslant k, m \leqslant T
    \beta n, | k - m | \leqslant R \sqrt{n}}  | Z^i_{k, m} | \right. &
    \left. > \sqrt{\varepsilon n} / R \right) +\mathbb{P} \left( \max_{0
    \leqslant k \leqslant nT} | Z^j_{0, k} | > \sqrt{\varepsilon n} R
    \right)\\
    \leqslant \quad  \frac{\sqrt{n}}{R} \mathbb{P} \left( \max_{0 \leqslant k
    \leqslant R \sqrt{n}}  | Z^i_{0, k} | > \sqrt{\varepsilon n} / R
    \right)  & +\mathbb{P} \left( \max_{0 \leqslant k \leqslant nT} | Z^j_{0,
    k} | > \sqrt{\varepsilon n} R \right) .
  \end{eqnarray*}
  Now, for any fixed $R$ the first term converges to zero by
  {\eqref{eq:two:pairs:estimate}} while the limsup of the right term is
  bounded by $\mathbb{P} \left( \max_{0 \leqslant t \leqslant T} | cW_t | >
  \sqrt{\varepsilon} R \right)$, which tends to zero as $R \rightarrow
  \infty$. Therefore we have proved the convergence is the uniform topology.

  To end, we shall now prove tightness of the $p / 2$-variation norms, $p >
  2$. As in the estimate for the first level, we observe that for $0 = t_0 <
  t_1 < \cdots < t_m = T$ and coordinates $1 \leqslant i, j \leqslant d$ we
  have
  \begin{eqnarray*}
    \sum_{r = 1}^m | A^{i, j}_{nt_{r - 1}, nt_r} | & = &  \sum_{0 \leqslant u
    < nT} \sum_{0 \leqslant r \leqslant m \of \kappa (nt_r) = u} | A^{i,
    j}_{nt_{r - 1}, nt_r} |\\
    & \lesssim &  \sum_{0 \leqslant u < nT} \# \{ 1 \leqslant r \leqslant m
    \of \kappa (nt_r) = u \}  | \Xi^i_u \Xi^j_u |\\
    & \leqslant &  \sum_{0 \leqslant u < nT} T_u  | \Xi^i_u \Xi^j_u |,
  \end{eqnarray*}
  which implies that $\mathbb{E} \left[ \frac{1}{n}  \left\|| \left(
  \sum_{\kappa (ns) \leqslant u \leqslant \kappa (nt)} A_{\tau_{u - 1}, \tau_u} (X)
  \right)_{0 \leqslant s < t \leqslant T} \right\||_{1, [0, T]} \right]
  \lesssim 1$. Also,
  \[ \begin{array}{lll}
       \left| \mathbb{X}_{ns, nt}^{(1)} -\mathbb{Z}_{\kappa (ns), \kappa (nt)}
       - \sum_{\kappa (ns) \leqslant u \leqslant \kappa (nt)} A_{\tau_{u - 1}, \tau_u} (X)
       \right| & \leqslant &  | \Xi_{\kappa (ns), \kappa (nt)}^{\otimes 2} | +
       | Z_{\kappa (ns), \kappa (nt)} | \otimes | \Xi_{\kappa (ns), \kappa
       (nt)} |
     \end{array}, \]
  and so
  \[
  \mathbb{E} \left[ \frac{1}{n}  \left\|| \left( \mathbb{X}_{ns,
  nt}^{(1)} -\mathbb{Z}_{\kappa (ns), \kappa (nt)} - \sum_{\kappa (ns) \leqslant u
  \leqslant \kappa (nt)} A_{\tau_{u - 1}, \tau_u} (X) \right)_{0 \leqslant s < t
  \leqslant T} \right\||_{1, [0, T]} \right] \lesssim 1.
  \]
  Observe that
  $\mathbb{E} [\| \mathbb{Z}^{(n)} \|_{p / 2, [0, T]}] \lesssim 1$, see
  {\eqref{eq:tightness p/2-var}}. Using the triangle inequality together with
  the fact that the $p$-variation norms are monotonically decreasing
  $\mathbb{E} [\| \mathbb{X}^{(n)} \|_{p / 2, [0, T]}] \lesssim 1$, as
  desired.
$\Box$

\ack
  The author is grateful to Tommaso Cornelis Rosati and Willem van Zuijlen for
  numerous valuable discussions. This work was supported by the German
  Research Foundation (DFG) via Research Unit FOR2402.

\appendix\section{Key renewal theorem}

In this section we show that the moment condition of Assumption
\ref{ass:jumps} is asymptotically equivalent to a second moment condition on
$\Xi_{\kappa (n)}$.

\begin{theorem}[Key renewal theorem]
  Assume that $p_j \geqslant 0, j \in \mathbb{N}_0, p_0 = 0, \sum_{j \in \mathbb{N}_0}
  p_j = 1$ so that $\mathrm{gcd} \{ j : p_j > 0 \} = d \in \mathbb{N}$. If
  $(b_n)_{n \in \mathbb{N}_0}$ is a summable sequence of non-negative real numbers,
  then the equation
  \begin{equation}
    a_n = \sum_{m = 0}^n b_{n - m} u_m
  \end{equation}
  \label{eq:renewal}has a unique solution satisfying
  \begin{equation}\label{eq:key renewal limit}
    \lim_{n \rightarrow \infty} \suppressed{\explicitspace} a_n =
    \frac{\sum_{m \in \mathbb{N}_0} b_m}{\sum_{j \in \mathbb{N}_0} jp_j},
  \end{equation}

  where $u_m = \sum_{k \in \mathbb{N}_0} p^{\ast k} (m)$ is the $k$-fold
  convolution of p evaluated at $m$, that is $p^{\ast k} (m) =\mathbb{P}
  \left( \sum_{j = 1}^k T_j = m \right)$, where $(T_k)_{k \geqslant 2}$ is a
  sequence of independent random variables so that $T_k, k \geqslant 2$ all
  have the same probability mass function $p$. Moreover, This to be understood
  even when $\sum_{j \in \mathbb{N}_0} jp_j = \infty$, in which case the limit on
  the right side of {\eqref{eq:key renewal limit}} is $0$.

  (To note about delay: $T_1$ has a different distribution.)
\end{theorem}

\begin{lemma}
  \label{lem:limiting moment}Let $\Xi_k$ as defined in Assumption
  \ref{ass:jumps}, then for $r, \ell \in \mathbb{N}_0$
  \[ \mathbb{E} [| \Xi_{\kappa (n)} |^{\otimes r} T_{\kappa (n)}^{\ell}]
     \xrightarrow[n \rightarrow \infty]{} \frac{\mathbb{E} [| \Xi_2 |^{\otimes
     r} T_2^{\ell + 1}]}{\mathbb{E} [T_2]}, \]
  whenever the right hand side is finite.
\end{lemma}

\begin{proof}
  Let $b_n =\mathbb{E} [| \Xi_2 |^{\otimes r} T_2^{\ell} 1_{T_2 > n}]$, then
  \[ \sum_{n \in \mathbb{N}_0} b_n = \sum_{n \in \mathbb{N}_0} \sum_{k > n} \mathbb{E}
     [| \Xi_2 |^{\otimes r} T_2^{\ell} 1_{T_2 = k}] = \sum_{k \in \mathbb{N}}
     k\mathbb{E} [| \Xi_2 |^{\otimes r} k^{\ell} 1_{T_2 = k}] =\mathbb{E} [|
     \Xi_2 |^{\otimes r} T_2^{\ell + 1}] . \]
  By The key renewal theorem there is a unique solution $(a_n)_{n \in \mathbb{N}_0}$
  to the equation {\eqref{eq:renewal}}, and it satisfies the limit in
  {\eqref{eq:key renewal limit}}. By the last computation the right hand side
  of {\eqref{eq:key renewal limit}} is exactly the one in the wanted
  assertion. It is therefore enough to show that $a_n =\mathbb{E} [|
  \Xi_{\kappa (n)} |^{\otimes r} T_{\kappa (n)}^{\ell}]$. Indeed,
  \begin{align*}
    &\mathbb{E} [| \Xi_{\kappa (n)} |^{\otimes r} T_{\kappa (n)}^{\ell}] =
    \sum_{k \in \mathbb{N}_0} \mathbb{E} [| \Xi_k |^{\otimes r} T_k^{\ell}
    1_{\kappa (n) = k}] = \sum_{k \in \mathbb{N}_0} \mathbb{E} [| \Xi_k |^{\otimes
    r} T_k^{\ell} 1_{\tau_k \leqslant n, \tau_k + T_k > n}]  \\
    &\quad = \sum_{k \in \mathbb{N}_0} \sum_{m = 0}^n \mathbb{E} [| \Xi_k |^{\otimes
    r} T_k^{\ell} 1_{T_k > n - m} 1_{\tau_k = m}] = \sum_{k \in \mathbb{N}_0}
    \sum_{m = 0}^n \mathbb{E} [| \Xi_k |^{\otimes r} T_k^{\ell} 1_{T_k > n -
    m}] \mathbb{P} [\tau_k = m]  \\
    &\quad = \sum_{m = 0}^n \mathbb{E} [| \Xi_2 |^{\otimes r} T_2^{\ell} 1_{T_2
    > n - m}]  \sum_{k \in \mathbb{N}_0} \mathbb{P} [\tau_k = m] = \sum_{m = 0}^n
    b_{n - m} u_m = a_n, 
  \end{align*}
  where in the fourth equality we used independence.
\end{proof}

Note that one could use the argument for $b_n$ of the form $\mathbb{E} [f (\{ X \tau_1,
\tau_1 + k \}_{0 \leqslant k \leqslant T_1}) g (T_1)]$, where
$f, g$ are real functions, as long as $(b_n)$ is an absolutely summable sequence (of well-defined
finite elements).

\section{Slutsky's Theorem in the Skorohod topology}

\begin{theorem}
  \label{thm:slutsky-skorohod}Assume that $X^n, Y^n \in D ([0, T],
  \mathbb{R})$ the Skorohod space of cadlag functions or that $X^n, Y^n \in C
  ([0, T], \mathbb{R})$) so that $X^n \rightarrow X$ in distribution in $D
  ([0, T], \mathbb{R})$, where $X \in C ([0, T], \mathbb{R})$ and $\| Y^n - f
  \|_{\infty, [0, T]} \rightarrow 0$ in probability, where $f$ is a
  deterministic continuous function. Then $X^n + Y^n \rightarrow X + f$ in
  distribution in $D ([0, T], \mathbb{R})$, or in
  $C\left(\left[0,T\right], \ensuremath{\mathbb{R}}\right)$ respectively.
\end{theorem}

\begin{proof}
  We shall cover the case $D = D ([0, T], \mathbb{R})$, the case of the
  uniform topology is similar. Let $\Phi \in C_b (D, \mathbb{R})$ and
  fix $\eta > 0$. The proof is completed if $\limsup_{n \rightarrow \infty} | \mathbb{E}
  [\Phi (X^n + Y^n) - \Phi (X + f)] | < \eta$. Define $\tilde{\Phi} (g)
  \assign \Phi (g + f)$. Note that $\tilde{\Phi} (g)$ is bounded and
  continuous in $D$ by the continuity of $f$. Now
  \begin{align*}
    & \limsup_{n \rightarrow \infty} | \mathbb{E} [\Phi (X^n + Y^n) - \Phi (X +
    f)] |    \\
    & \quad \leqslant \limsup_{n \rightarrow \infty} | \mathbb{E} [\Phi (X^n +
    Y^n) - \Phi (X^n + f)] | + \limsup_{n \rightarrow \infty} | \mathbb{E}
    [\Phi (X^n + f) - \Phi (X + f)] |  \\
    & \quad = \limsup_{n \rightarrow \infty} | \mathbb{E} [\Phi (X^n + Y^n) -
    \Phi (X^n + f)] | + \limsup_{n \rightarrow \infty} | \mathbb{E}
    [\tilde{\Phi} (X^n) - \tilde{\Phi} (X)] |  \\
    & \quad = \limsup_{n \rightarrow \infty} | \mathbb{E} [\Phi (X^n + Y^n) -
    \Phi (X^n + f)] | .
  \end{align*}
  However, since the Skorohod distance $d$  on $D$ is controlled by the uniform distance
  which is homogeneous, we have $d (g + h, f + h) \leqslant \| f - g
  \|_{\infty, [0, T]}$ for any $h$,
  we have
  \begin{align*}
    & | \mathbb{E} [\Phi (X^n + Y^n) - \Phi (X^n + f)] |  \\
    & \: = | \mathbb{E} [(\Phi (X^n + Y^n) - \Phi (X^n + f)) (1_{d (X^n + Y^n,
    X^n + f) \geqslant \varepsilon} + 1_{d (X^n + Y^n, X^n + f) <
    \varepsilon})] |  \\
    & \: \leqslant 2 \| \Phi \|_{\infty} \mathbb{P} [d (X^n + Y^n, X^n + f)
    \geqslant \varepsilon] + | \mathbb{E} [(\Phi (X^n + Y^n) - \Phi (X^n + f))
    1_{d (X^n + Y^n, X^n + f) < \varepsilon}] |  \\
    & \: = 2 \| \Phi \|_{\infty} \mathbb{P} [\| Y^n - f \|_{\infty, [0, T]}
    \geqslant \varepsilon] + | \mathbb{E} [(\Phi (X^n + Y^n) - \Phi (X^n + f))
    1_{d (X^n + Y^n, X^n + f) < \varepsilon}] |  \\
    & \: \simeq | \mathbb{E} [(\Phi (X^n + Y^n) - \Phi (X^n + f)) 1_{d (X^n +
    Y^n, X^n + f) < \varepsilon}] | .
  \end{align*}
  By tightness of $X^n$ we have a compact $K_{\eta}$ so that $\mathbb{P} (X^n
  \nin K_{\eta}) < \frac{\eta}{4 \| \Phi \|_{\infty}}$. Note that also
  $\tilde{K}_{\eta} \assign K_{\eta} + f : = \{ g + f : g \in K_{\eta} \}$ is
  compact. Since $\Phi$ is continuous it is equicontinuous on
  $\tilde{K}_{\eta}$. Therefore, there exists some $\varepsilon_0 > 0$ so that
  if $v \in \tilde{K}_{\eta}$ and $u \in D$ are so that $d (u, v) <
  \varepsilon_0$, then$| \Phi (u) - \Phi (v) | < \eta / 2$. Note that $X^n \in
  K_{\eta}$ if and only if $X^n + f \in \tilde{K}_{\eta}$ and therefore, if we
  choose $\varepsilon \leqslant \varepsilon_0$ we have
  \begin{align*}
   & | \mathbb{E} [(\Phi (X^n + Y^n) - \Phi (X^n + f)) 1_{d (X^n + Y^n, X^n +
    f) < \varepsilon}] |  \\
   & \quad \leqslant 2 \| \Phi \|_{\infty}  | \mathbb{P} (X^n \nin
    K_{\varepsilon}) | + | \mathbb{E} [(\Phi (X^n + Y^n) - \Phi (X^n + f))
    1_{d (X^n + Y^n, X^n + f) < \varepsilon, X^n + f \in \tilde{K}_{\eta}}] |
     \\
    & \quad < \eta / 2 + \eta / 2\mathbb{P} (d (X^n + Y^n, X^n + f) <
    \varepsilon, X^n + f \in \tilde{K}_{\eta}) < \eta .
  \end{align*}
\end{proof}

\section{Rough Donsker Theorem}\label{app:donsker}
Let $Z_k =
\sum_{\ell = 1}^k Y_\ell, k\in\mathbb{N}$ be a delayed random walk on $\mathbb{R}^d$ starting at $Z_0=0$, that is $(Y_\ell)_{\ell \in \mathbb{N}}$ are independent and have the same distribution for $\ell\geqslant 2$. Assume that $\mathbb{E} [Y_\ell] = 0$ and $\mathbb{E} [| Y_\ell |^2] <\infty$ for $\ell = 1, 2$.
Let $Z^{(n)}$ and $\mathbb{Z}^{(n)}$ be defined as in
{\eqref{eq:interpolation}}, then \eqref{eq:IPforZ} holds.

\begin{proof}
  The convergence in distribution $Z^{(n)} \xrightarrow[n \rightarrow
  \infty]{} B^Z$ in the uniform topology on $C ([0, T] ;
  \mathbb{R}^d)$ is Donsker's Theorem, cf.{\tmabbr{}}
  \cite{Don51,Bil99}. Since $\mathbb{E} [|
  (Z^{(n)}_T)^{\otimes 2} |] = T\mathbb{E} [| Z_1^{\otimes 2} |] < \infty$,
  the conditions of Theorem 2.2 of Kurtz-Protter {\cite{KP91}} are fulfilled
  (and particularly their Example 3.1 with $m = 2$) and therefore we conclude
  the joint convergence of $(\hat{Z}^{(n)}, \hat{\mathbb{Z}}^{(n)})$ to
  $\left(B^Z, \mathbb{B}^{\mathrm{It{\hat o}}}\right)$ in the Skorohod topology on
  the space $D ([0, T] ; \mathbb{R}^d \times \mathbb{R}^{d \times d})$ of
  c{\`a}dl{\`a}g functions $\nobracket [0, T] \rightarrow \mathbb{R}^d \times
  \mathbb{R}^{d \times d})$, where $\hat{Z}^{(n)}_t = \frac{1}{\sqrt{n}}
  Z_{\lfloor tn \rfloor}$, $\hat{\mathbb{Z}}^{(n)}_t = \frac{1}{n}  \sum_{r =
  1}^{\lfloor nt \rfloor} Z_{r - 1} \otimes Z_{r - 1, r}$, and
  $\mathbb{B}^{\mathrm{It{\hat o}}}$ is the It{\^o} iterated integral of $B$. Notice that
  \[ \mathbb{Z}^{(n)}_t - \hat{\mathbb{Z}}^{(n)}_t = \frac{1}{2 n}  \sum_{r =
     1}^{\lfloor nt \rfloor} Z_{r - 1, r}^{\otimes 2} - \frac{nt - \lfloor nt
     \rfloor}{2 n} Z_{\lfloor nt \rfloor, \lfloor nt \rfloor + 1}^{\otimes 2}
  \]
  which converges to $\mathbb{E} [Y_2 \otimes Y_2]  \frac{t}{2} = \frac{1}{2}
  [B^Z]_t$ in probability in the uniform norm on $[0, T] $ by the ergodic
  theorem and the integrability of the limit. Indeed, on the
  diagonal
  \[ \mathbb{E} \| (\hat{\mathbb{Z}}^{(n)} -\mathbb{Z}^{(n)})^{i, i} \|_{1,
     [0, T]} =\mathbb{E} [| (\hat{\mathbb{Z}}^{(n)}_T -\mathbb{Z}^{(n)}_T)^{i,
     i} |] \leqslant \frac{T}{2}  ([Y_1^2] +\mathbb{E} [Y_2^2]) < \infty .
     \]
  The point-wise limit is
  identified by the ergodic theorem, and altogether this gives
  \[ \left\| (\hat{\mathbb{Z}}^{(n)} -\mathbb{Z}^{(n)})^{i, i}
     -\mathbb{E}[Y_2 \otimes Y_2] \frac{\cdummy}{2} \right\|_{\infty, [0, T]}
     \xrightarrow[n \rightarrow \infty]{} 0 \text{ in probability} . \]
  The off-diagonal terms then follow by polarization. It is straight-forward
  to conclude then the joint convergence of $(Z^{(n)}_t,
  \mathbb{Z}^{(n)}_t)_{0 \leqslant t \leqslant T}$ to $(B^Z, \mathbb{B}^Z) =
  \left( B^Z, \mathbb{B}^{\mathrm{It{\hat o}}} + \frac{1}{2} [B^Z,B^Z]_t \right)$ in
  the Skorohod topology on the space $C ([0, T] ; \mathbb{R}^d \times
  \mathbb{R}^{d \times d})$.

  By Lemma \ref{lem:tightness of norms is enough} together with
  {\eqref{eq:p-var norm of pair}} we only need to show the tightness of the
  sequence of random variables $\| Z^{(n)} \|_{p, [0, T]}$ and $\|
  \mathbb{Z}^{(n)} \|_{p / 2, [0, T]}$, for all $p > 2$. For the first level
  note first that
  \begin{equation}\label{eq:tightness p-var first level}
    \mathbb{E} [\| Z^{(n)} \|^2_{p, [0, T]}] \lesssim [[Z^{(n)}, Z^{(n)}]_T]
    \lesssim T (\mathbb{E} [Y_2^2] +\mathbb{E} [Y_1^2])
  \end{equation}
  by L{\'e}pingle's $p$-variation
  inequality {\cite{Le76}} combined with the BDG inequality. For tightness of the second level, we claim first that for the
  It{\^o} sum we have
  \[ \mathbb{E} [\| \hat{\mathbb{Z}}^{(n)} \|_{p / 2, [0, T]}] \lesssim
     [[Z^{(n)}, Z^{(n)}]_T] \lesssim T (\mathbb{E} [Y_2^2] +\mathbb{E}
     [Y_1^2]) . \]
  Indeed, this follows from Theorem 1.1 of {\cite{Zor20}}, which is an off-diagonal
  version to the L{\'e}pingle $p$-variation BDG inequality
  (one can also use \cite{FrZor20} or Proposition 3.8 of {\cite{DOP}}). Since
  \[ \text{} \mathbb{E} [\| \mathbb{Z}^{(n)}_t - \hat{\mathbb{Z}}^{(n)}_t
     \|_{1, [0, T]}] \leqslant \frac{T}{2}  ([Y_1^2] +\mathbb{E} [Y_2^2]) <
     \infty, \]
  we conclude that
  \begin{eqnarray}\label{eq:tightness p/2-var}
    \mathbb{E} [\| \mathbb{Z}^{(n)} \|_{p / 2, [0, T]}] & \lesssim &
    \mathbb{E} [\| \hat{\mathbb{Z}}^{(n)} \|_{p / 2, [0, T]}] +\mathbb{E} [\|
    \mathbb{Z}^{(n)} - \hat{\mathbb{Z}}^{(n)} \|_{p / 2, [0, T]}] \nonumber\\
    & \lesssim & T (\mathbb{E} [Y_2^{\otimes 2}] +\mathbb{E} [Y_1^{\otimes
    2}]) \suppressed{,}
  \end{eqnarray}
  and the proof is complete.
\end{proof}

\end{document}